\documentclass[12pt]{amsart}
\usepackage{amsfonts}
\numberwithin{equation}{section}
\usepackage{amssymb}
\usepackage{amsmath}

\def\arg{\operatorname{arg}}

\newtheorem{lem}{Lemma}[section]
\newtheorem{thm}{Theorem}[section]

\newtheorem{defi}{Definition}[section]
\theoremstyle{remark}

\title[Borel radius]{Common Borel radius of an algebroid function and its derivative}
\thanks{The work is supported by NSF of China (No.10871108)}
\author[Wu]{Wu Nan$^{1}$}
\date{\today, Preliminary version}
\author[Xuan]{Xuan Zu-xing$^{1,2}$}
\address{$^{1}$Department of Mathematical Sciences}\address{Tsinghua University}\address{Beijing,
100084}\address{People's Republic of China}
\email{wunan07@gmail.com}
\address{$^{2}$Basic Department}\address{Beijing Union University}\address{No.97 Bei Si Huan Dong
Road}\address{Chaoyang District}\address{Beijing,
100101}\address{People's Republic of China}
\email{xuanzuxing@ss.buaa.edu.cn}

\begin{document}

\begin{abstract}

In this article, by comparing the characteristic functions, we prove
that for any $\nu$-valued algebroid function $w(z)$ defined in the
unit disk with
$\limsup_{r\rightarrow1-}T(r,w)/\log\frac{1}{1-r}=\infty$ and the
hyper order $\rho_2(w)=0$, the distribution of the Borel radius of
$w(z)$ and $w'(z)$ is the same. This is the extension of G.
Valiron's conjecture for the meromorphic functions defined in
$\widehat{\mathbb{C}}$.
\end{abstract}


\maketitle

\section{Introduction and Main Results}
\setcounter{equation}{0}

The value distribution theory of meromorphic functions due to R.
Nevanlinna(see \cite{Hayman} for standard references) was extended
to the corresponding theory of algebroid functions by H. Selberg
\cite{He}, E. Ullrich \cite{U} and G. Valiron \cite{V1} around 1930.
The singular direction for $w(z)$ is one of the main objects studied
in the theory of value distribution of algebroid functions. Several
types of singular directions have been introduced in the literature.
Their existence and some connections between them have also been
established\ \cite{R,T,V2.}.

In 1928, G. Valiron \cite{V3} asked the following:

 \emph{Does there exist a common
Borel direction of a meromorphic function and its derivative?}

This question was investigated by many mathematicians, such as
G.Valiron \cite{V4}, A.Rauch \cite{R2}, C.T. Chuang \cite{Chuang}.
They proved the existence of common Borel directions under some
conditions. However, it is still an open problem till now. For the
case of the unit disk, Zhang \cite{Zhang} solved the problem, he
proved that the Borel radius of a meromorphic function of finite
order is the same as its derivative. We associated it to the
algebroid functions and ask weather the Borel radius of a
$\nu-$valued algebroid function is the same to its derivative. To
state our results clearly, we begin with some basic nations for
algebroid functions.

Let $w=w(z) (z\in \Delta)$ be the $\nu$-valued algebroid function
defined by the irreducible equation
\begin{equation}\label{1.1}
A_\nu(z)w^\nu+A_{\nu-1}(z)w^{\nu-1}+\cdots+A_0(z)=0,
\end{equation}
where $A_\nu(z),...,A_0(z)$ are analytic functions without any
common zeros. The single-valued domain $\widetilde{R}_z$ of
definition of $w(z)$ is a $\nu$-valued covering of the $z$-plane and
it is a Riemann surface.

A point in $\widetilde{R}_z$ is denoted by $\widetilde{z}$ if its
projection in the $z$-plane is $z$. The open set which lies over
$|z|<r$  is denoted by $|\widetilde{z}|<r$.  Let $n(r,a)$ be the
number of zeros, counted according to their multiplicities, of
$w(z)-a$ in $|\widetilde{z}|\leq r,$ $n(r,a)$ be the number of
distinct zeros of $w(z)-a$ in $|\widetilde{z}|\leq r.$ Let
\begin{eqnarray*}
 N(r,a)&=&\frac{1}{\nu}\int_0^r\frac{n(t,a)-n(0,a)}{t}d
 t+\frac{n(0,a)}{\nu}\log{r},\\
 m(r,a)&=&\frac{1}{2\pi\nu}\int_{|\widetilde{z}|=r}\sum\limits_{j=1}^\nu\log^+|\frac{1}{w_j(r e^{i\theta})-a}|d \theta,\ \
 z=re^{i\theta},\\
 T(r,a)&=&m(r,a)+N(r,a).
\end{eqnarray*}
 where $|\widetilde{z}|=r$ is the boundary of
$|\widetilde{z}|\leq r$. Moreover, $S(r,w)$ is a conformal invariant
and is called the mean covering number of $|\widetilde{z}|\leq r$
into $w$-sphere. We call $T(r,w)=T(r,\infty)$ the characteristic
function of $w(z)$. It is known from [\cite{He}, $3^o$, p.84] that
$T(r,a)=m(r,\infty)+N(r,\infty)+O(1).$ We define the order and hyper
order of a $\nu$-valued algebroid function as
$$\rho(w)=\limsup\limits_{r\rightarrow1-}\frac{\log T(r,w)}{\log\frac{1}{1-r}},$$
and
$$\rho_2(w)=\limsup\limits_{r\rightarrow1-}\frac{\log\log T(r,w)}{\log\frac{1}{1-r}}.$$

Given an angular domain
$$\Delta(\theta_0,\varepsilon)=\{z||\arg z-\theta_0|<\varepsilon\},0<\varepsilon<\frac{\pi}{2},$$
we denote $\{z:|z|<r, |\arg z-\theta|<\varepsilon\}$ by
$\Omega(r,\theta,\varepsilon)$ and write $\widetilde{\Omega}$ for
the part of $\widetilde{R}_z$ on $\Omega(r,\theta,\varepsilon)$.
$\overline{n}(r,\Delta(\theta,\varepsilon),w=a)$ denotes the numbers
of $w(z)-a$ in $\widetilde{\Omega}$(not counting multiplicities).

\begin{eqnarray*}
\overline{N}(r,\Delta(\theta,\varepsilon),w=a)&=&\frac{1}{\nu}\int_0^r\frac{\overline{n}(t,\Delta(\theta,\varepsilon),w=a)-\overline{n}(0,\Delta(\theta,\varepsilon),w=a)}{t}dt\\
&+&\frac{\overline{n}(0,\Delta(\theta,\varepsilon),w=a)}{\nu}\log r
\end{eqnarray*}
is called the counting function of zeros of $w(z)-a$ in $\Omega$.

Next, we give the definition of the Borel radius of a $\nu$-valued
algebroid function in the unit disk.
\begin{defi}
A radius $L(\theta): \arg z=\theta, 0<|z|<1$ is called a Borel
radius of a $\nu$-valued algebroid function $w(z)$ of order $\rho$,
if for any $\varepsilon>0$
$$\limsup\limits_{r\rightarrow1-}\frac{\log \overline{N}(r,\Delta(\theta,\varepsilon),w=a)}{\log\frac{1}{1-r}}=\rho$$
holds for any $a\in\mathbb{\widehat{C}}$, except for $2\nu$
exceptions.
\end{defi}

In this note, we give a positive answer to the G. Valiron's
conjecture for algebroid functions defined in the unit disk.
\begin{thm}\label{thm1.1}
The distribution of the Borel radius of a $\nu$-valued algebroid
function $w(z)$ with the order $0\leq\rho(w)<\infty$ and
$$ \limsup\limits_{r\rightarrow1-}\frac{T(r,w)}{\log\frac{1}{1-r}}=\infty $$
is the same to that of its derivative.
\end{thm}
\begin{thm}\label{thm1.2}
The distribution of the Borel radius of a $\nu$-valued algebroid
function $w(z)$ with order $\rho(w)=\infty$ and the hyper order
$\rho_2(w)=0$ is the same to that of its derivative.
\end{thm}

We will prove the above two theorems synchronously.

\section{Primary knowledge}

\begin{lem}
Let $w(z)$ be the $\nu$-valued algebroid function defined by
\eqref{1.1} in the unit disk, $z=z(\zeta)$ be a conformal mapping
from the unit disk $D(\zeta)$ into $D(z)$.  Then
$M(\zeta)=w(z(\zeta))$ and $M'(\zeta)$ are also $\nu$-valued
algebroid functions. Furthermore, we can see that
$G(\zeta)=w(z(\zeta))$ is determined by
\begin{equation*}
A_\nu(z(\zeta))M^\nu(\zeta)+A_{\nu-1}(z(\zeta))M^{\nu-1}(\zeta)+\cdots+A_0(z(\zeta))=0,
\end{equation*}
and $M'(\zeta)=w'(z(\zeta))z'(\zeta)$.
\end{lem}
Lemma 2.1 is apparent and we omit the proof of it. The following
lemma is an analogue of Lemma 2.1 in \cite{Zhang}.
\begin{lem}\label{01}
Set
$$G(r,\theta,\eta)=\{z:0<|z|<r,|\arg z-\theta|<\eta\},$$
$$\alpha=\frac{\pi}{2\eta},$$
$$\zeta(z)=\frac{(ze^{-i\theta})^{2\alpha}+2(ze^{-i\theta})^\alpha-1}{(ze^{-i\theta})^{2\alpha}-2(ze^{-i\theta})^\alpha-1}.$$
The function $\zeta=\zeta(z)$ defined above maps conformally the
unit disk $D(\zeta)=\{\zeta:|\zeta|<1\}$ onto the sector
$G(1,\theta,\eta)$. By $z=z(\zeta)$ we denote the inverse function
of the function $\zeta(z)$. Write $M(\zeta)=w(z(\zeta))$, where
$w(z)$ is a $\nu-$valued algebroid function in the sector
$G(1,\theta,\eta)$. Then for any value $a$ on the complex plane, we
have

(1) Set $\beta=2^{-\alpha-\frac{5}{2}}$. Then
$$\overline{N}(r,\Delta(\theta,\frac{\eta}{2}),w=a)\leq\frac{2}{\beta}\overline{N}(1-\beta(1-r), M=a)+O(1),$$
when $r\rightarrow1-$.

(2) Set $\delta=\frac{1}{16\alpha}$. Then
$$\overline{N}(\gamma,M=a)\leq\frac{2}{\delta}\overline{N}(1-\delta(1-\gamma),\Delta(\theta,\eta), w=a)+O(1),$$
when $\gamma\rightarrow1-$.

(3) For any $0<t<1$, we have
\begin{equation}\label{05}
T(t,z'(\zeta))\leq3\log\frac{2}{1-t},\ \
T(t,\frac{1}{z'(\zeta)})\leq3\log\frac{2}{1-t}+\log\frac{\pi}{\eta}.
\end{equation}
\end{lem}
Here we generalize the corresponding results of meromorphic
functions to algebroid functions. This lemma for meromorphic
functions was first established by Zhang in \cite{Zhang1}. He proved
that the function $\zeta=\zeta(z)$ maps the unit disk
$D(\zeta)=\{\zeta:|\zeta|<1\}$ onto the sector $G(1,\theta,\eta)$
conformally. Furthermore, after a calculation Zhang found that this
function has the following perfect properties:
\begin{equation}\label{2.5}\zeta(\{z: \frac{1}{2}<|z|<r, |\arg
z-\theta|<\frac{\eta}{2}\})\subset \{\zeta:
|\zeta|<1-2^{-\frac{\pi}{2\eta}-\frac{\pi}{2}}(1-r)\} \end{equation}
and \begin{equation}\label{2.6} z(\{\zeta: |\zeta|<\gamma\})\subset
\{z:|z|<1-\frac{\eta}{8\pi}(1-\gamma), |\arg z-\theta|<\eta\}.
\end{equation} This is important. The number of roots of algebroid
functions or meromorphic functions are conformal invariant
consequently he obtained this result.

\textbf{Remark.} As we know that the term $T(r,\Omega,f)$, whose
definition can be seen in Page 233 of \cite{Tsuji} is conformal
invariant, where $f$ is a meromorphic function in the angular domain
$G(1,\theta,\eta)$. By \eqref{2.5} and \eqref{2.6} we have the
following
$$T(r,\Delta(\theta,\frac{\eta}{2}),f(z))\leq T(1-2^{-\frac{\pi}{2\eta}-\frac{\pi}{2}}(1-r),f(z(\zeta)))$$
and
$$T(\gamma, f(z))\leq T(1-\frac{\eta}{8\pi}(1-\gamma),\Delta(\theta,\eta),f(z(\zeta))).$$
From the above we can see that the order of $T(r,f(z(\zeta)))$ is
$\rho$ in the unit disk if and only if there exists a $\varepsilon$
such that
\begin{equation}\label{2.4}
\limsup\limits_{r\rightarrow1-}\frac{\log
T(r,\Delta(\theta,\varepsilon),f)}{\log\frac{1}{1-r}}=\rho.
\end{equation}
Since $L(\theta)$ is a Borel radius of a meromorphic function $f$ in
the unit disk if and only if there exists a $\varepsilon$ such that
\eqref{2.4} holds. Therefore, we can simplify the Zhang's proof for
$L(\theta)$ is a Borel radius if and only if the order of
$T(r,f(z(\zeta)))$ is $\rho$.

\begin{lem}\label{02}
Let $h(r)$ is a real non-negative and non-decreasing function
defined in $(0,1)$, $E\subset(0,1)$ is a set with
$\int_E\frac{1}{1-r}dr<\infty$. If
\begin{equation}\label{2.2}
\limsup\limits_{r\rightarrow1-}\frac{\log
h(r)}{\log\frac{1}{1-r}}=\rho,
\end{equation}
then we have \begin{equation}\label{2.3}\limsup\limits_{r\notin
E,r\rightarrow1-}\frac{\log h(r)}{\log\frac{1}{1-r}}=\rho.
\end{equation}
\end{lem}
Now we give the proof of Lemma \ref{02}.
\begin{proof}
If $\rho=0$, it is easy to see that the conclusion naturally holds.
Here we only consider the case $0<\rho\leq\infty$.

We choose a $0<\lambda<1$ such that
$$\log \frac{1}{\lambda}>K_E,$$ where $K_E=\int_E\frac{dr}{1-r}<\infty$.
Suppose \eqref{2.3} is not true, then there exists a number
$0<\rho_1<\rho$, such that $$\limsup\limits_{r\notin
E,r\rightarrow1^{-}}\frac{\log
h(r)}{\log\frac{1}{1-r}}=\rho_1<\rho.$$ From \eqref{2.2}, we can
take a sequence $\{r_n\}\subset(r_0,1)$ with $r_n\rightarrow1-$ such
that
\begin{equation}\label{set01+}
\limsup\limits_{n\rightarrow\infty}\frac{\log
h(r_n)}{\log\frac{1}{1-r_n}}=\rho.
\end{equation} Since for each $n$
\begin{equation*}
\begin{split}
\int_{[r_n,\lambda r_n+(1-\lambda)]\backslash
E}\frac{dr}{1-r}&\geq\int_{[r_n,\lambda
r_n+(1-\lambda)]}\frac{dr}{1-r}-\int_E\frac{dr}{1-r}\\
&=\log\frac{1}{\lambda}-K_E>0,
\end{split}
\end{equation*}
there exists a $r_n'\in[r_n,\lambda r_n+(1-\lambda)]\backslash E$.
By the increasing property of $\log h(r)$, we have
$$\frac{\log h(r_n')}{\log\frac{1}{1-r_n'}}\geq\frac{\log h(r_n)}{\log\frac{1}{\lambda(1-r_n)}}
=\frac{\log h(r_n)}{\log\frac{1}{\lambda}+\log\frac{1}{1-r_n}},$$
and then we have
\begin{equation*}
\begin{split}
\limsup\limits_{n\rightarrow\infty}\frac{\log
h(r_n)}{\log\frac{1}{1-r_n}}&
=\limsup\limits_{n\rightarrow\infty}\frac{\log h(r_n)}{\log\frac{1}{\lambda}+\log\frac{1}{1-r_n}}\\
&\leq \limsup\limits_{r_n'\rightarrow1-}
\frac{\log h(r_n')}{\log\frac{1}{1-r_n'}}\\
&\leq \limsup\limits_{r\rightarrow1-,\ r\in[r_0,1]\backslash
E}\frac{\log h(r)}{\log\frac{1}{1-r}}=\rho_1<\rho.
\end{split}
\end{equation*}
This contradicts to (\ref{set01+}). Our Lemma is confirmed.
\end{proof}
In 1988, Zeng \cite{Zeng} established the following lemma which is a
classical result for algebroid functions and is useful for our
study.
\begin{lem}\cite{Zeng}\label{lem2.3}
Let $w(z)$ be the $\nu$-valued algebroid function defined by (1.1),
then $w'(z)$ is also a $\nu$-valued algebroid function in the unit
disk and $\rho(w)=\rho(w')$.
\end{lem}
The following lemma is the second fundamental theorem for algebroid
functions in the unit disk, whose proof can be seen in \cite{He},
and we can obtain the error term $S(r,w)$ by the same method as used
in meromorphic functions.
\begin{lem}
Let $w(z)$ be a $\nu-$valued algebroid function in the unit disk,
and $a_1,a_2,\cdots,a_{q}$ be $q$ different values on the complex
sphere, then we have
$$(q-2\nu)T(r,w)<\sum\limits_{i=1}^q\overline{N}(r,w=a_i)+S(r,w),$$
where
\begin{equation*}
S(r,w)=\begin{cases} O(\log\frac{1}{1-r}) &\text{,if $\lambda(w)<\infty$},\\
O(\log\frac{1}{1-r}+\log T(r,w)),r\notin E & \text{,if
$\lambda(w)=\infty$}.\end{cases}
\end{equation*} where $E$ is a set such that $E\subset(0,1)$ and
$\int_E\frac{1}{1-r}dr<\infty$.

In general, we can write the second fundamental theorem as follows
$$(q-2\nu)T(r,w)<\sum\limits_{i=1}^q\overline{N}(r,w=a_i)+O(\log\frac{1}{1-r}+\log T(r,w)),r\notin E.$$
\end{lem}
\section{Main lemma}
Now we are in the position to show our main lemma which is crucial
to our theorems.
\begin{lem}\label{03}
Let $w(z)$ be a $\nu-$valued algebroid function of order
$\rho(w)=\rho$ ($0\leq\rho\leq\infty$) ,
$\limsup_{r\rightarrow1-}T(r,w)/\log\frac{1}{1-r}=\infty$ and
$\rho_2(w)=0$ in the unit disc $D(z)$. Then a $radius$ $L(\theta)$
is a Borel radius of the algebroid function $w(z)$ if and only if
for any $0<\eta<1$, the function $M(\zeta)=w(z(\zeta))$ is a
$v-$valued algebroid function of order $\rho$ in the unit disk
$D(\zeta)$, where $z=z(\zeta)$ is the function described in Lemma
\ref{01}, mapping the unit disk $D(\zeta)$ onto the sector
$G(1,\theta,\eta)$.
\end{lem}
\begin{proof}$"\Longrightarrow"$

Let $L(\theta)$ be a Borel radius of the function $w(z)$. Then for
any fixed $0<\eta<1$, there exist $2\nu+1$ different values
$a_1,\cdots, a_{2\nu+1}$ on the complex plane, such that
$$\limsup\limits_{r\rightarrow1-}\frac{\log \overline{N}(r,\Delta(\theta,\varphi),w=a_i)}{\log\frac{1}{1-r}}=\rho, (i=1,2,\cdots, 2\nu+1; \varphi=\eta,\frac{\eta}{2}).$$
Applying Lemma \ref{01} to the function $w(z)$, we have

\begin{equation*}
\begin{split}
\limsup\limits_{\gamma\rightarrow1-}\frac{\log \overline{N}(\gamma,
M=a_i)}{\log\frac{1}{1-\gamma}}&=\limsup\limits_{r\rightarrow1-}\frac{\log\frac{2}{\beta}
\overline{N}(1-\beta(1-r), M=a_i)}{\log\frac{1}{1-(1-\beta(1-r))}}\\
&\geq\limsup\limits_{r\rightarrow1-}\frac{\log
\overline{N}(r, \Delta(\theta,\frac{\eta}{2}), w=a_i)}{\log\frac{1}{1-r}}=\rho(i=1,2,\cdots,2\nu+1).\\
\end{split}
\end{equation*}
Therefore the order of the function $M(\zeta)$ is not less than
$\rho$. Apply Lemma \ref{01} to the function $w(z)$, we have
\begin{equation*}
\begin{split}
\limsup\limits_{\gamma\rightarrow1-}\frac{\log \overline{N}(\gamma,
M=a_i)}{\log\frac{1}{1-\gamma}}&\leq\limsup\limits_{\gamma\rightarrow1-}\frac{\log\frac{2}{\delta}
\overline{N}(1-\delta(1-\gamma), \Delta(\theta,\eta), w=a_i)}{\log\frac{1}{1-(1-\delta(1-\gamma))}}\\
&=\limsup\limits_{r\rightarrow1-}\frac{\log
\overline{N}(r,\Delta(\theta,\eta), w=a_i)}{\log\frac{1}{1-r}}=\rho(i=1,2,\cdots,2\nu+1).\\
\end{split}
\end{equation*}
Applying the second fundamental theorem to the function $M(\zeta)$.
We obtain
$$T(\gamma,M)\leq\sum\limits_{i=1}^{2\nu+1}\overline{N}(\gamma,M=a_i)+O(\log\frac{1}{1-\gamma}+\log T(\gamma,M)), \gamma\notin E,$$
where $E$ is a set with $\int_E\frac{1}{1-\gamma}d\gamma<\infty$.
Hence
$$\limsup\limits_{\gamma\notin E, \gamma\rightarrow1-}\frac{\log T(\gamma,M)}{\log\frac{1}{1-\gamma}}
\leq\limsup\limits_{\gamma\notin E, \gamma\rightarrow1-}\frac{\log
\sum\limits_{i=1}^{2\nu+1}\overline{N}(\gamma,M=a_i)}{\log\frac{1}{1-\gamma}}=\rho.$$
Applying Lemma \ref{02}, we can see that the order of the function
$G(\zeta)$ is $\rho$.

$ "\Longleftarrow"$ Now for any fixed $0<\eta<1$, let
$M(\zeta)=w(z(\zeta))$ be a $\nu-$valued algebroid function of order
$\rho$ in the unit disk $D(\zeta)$, where $z=z(\zeta)$ is the
mapping function defined in Lemma \ref{01}. Then for any $2\nu+1$
different values $a_1,a_2,\cdots,a_{2\nu+1}$, applying the second
fundamental theorem, we have
\begin{equation*}
\begin{split}
T(\gamma,M)&\leq\sum\limits_{i=1}^{2\nu+1}\overline{N}(\gamma,M=a_i)+O(\log\frac{1}{1-\gamma}+\log T(\gamma,M))\\
&\leq\sum\limits_{i=1}^{2\nu+1}\frac{2}{\delta}\overline{N}(1-\delta(1-\gamma),\Delta(\theta,\eta), w=a_i)+O(\log\frac{1}{1-\gamma}+\log T(\gamma,M))\\
\end{split}
\end{equation*}
hence by Lemma \ref{02}
\begin{equation*}
\begin{split}
\rho&=\limsup\limits_{\gamma\notin E,\gamma\rightarrow1-}\frac{\log
T(\gamma,M)}{\log\frac{1}{1-\gamma}}\leq\limsup\limits_{\gamma\rightarrow1-}\frac{\log\sum\limits_{i=1}
^{2\nu+1}\overline{N}(1-\delta(1-\gamma),\Delta(\theta,\eta),w=a_i)}{\log\frac{1}{1-(1-\delta(1-\gamma))}}\\
&=\limsup\limits_{r\rightarrow1-}\frac{\log\sum\limits_{i=1}
^{2\nu+1}\overline{N}(r,\Delta(\theta,\eta),w=a_i)}{\log\frac{1}{1-r}}\leq\limsup\limits_{r\rightarrow1-}\frac{\log\sum\limits_{i=1}
^{2\nu+1}\overline{N}(r,w=a_i)}{\log\frac{1}{1-r}}=\rho.
\end{split}
\end{equation*}
Thus $L(\theta)$ is a Borel radius of the function $w(z)$.
\end{proof}

\section{Proof of the theorems}
Suppose that $w(z)$ is a $v-$valued algebroid function of order
$\rho$ in the unit disk $D(z)$ and $L(\theta)$ be a Borel radius of
$w(z)$. For any $0<\eta<1$, we write $M(\zeta)=w(z(\zeta))$, where
$z=z(\zeta)$ is the function in Lemma \ref{01}. Since
$M'(\zeta)=w'(z(\zeta))z'(\zeta)$, we have
$$T(t,M'(\zeta))\leq T(t, w'(z(\zeta)))+T(t,z'(\zeta))$$
$$T(t,w'(z(\zeta)))\leq T(t, M'(\zeta))+T(t,\frac{1}{z'(\zeta)})=T(t, M'(\zeta))+T(t,z'(\zeta))+O(1).$$
Combining the above two inequalities and noting Lemma \ref{01}, we
have
\begin{equation}
|T(t,M'(\zeta))-T(t,w'(z(\zeta))|\leq|T(t,z'(\zeta))|\leq3\log\frac{2}{1-t}+\log\frac{\pi}{\eta}.
\end{equation}
By Lemma \ref{lem2.3}, we can see that $\rho(M')=\rho(M)=\rho$.
Therefore the order of the function $w'(z(\zeta))$ is also $\rho$.
Then by Lemma \ref{03}, $L(\theta)$ is also a Borel radius of the
function $w'(z)$.

Next we suppose that $L(\theta)$ is a Borel radius of the function
$w'(z)$. By Lemma \ref{03}, the function $w'(z(\zeta))$ is an
algebroid function of order $\rho$ in the unit disk $D(\zeta)$. Then
the order of the function $M(\zeta)$ is also $\rho$. Moreover, we
use Lemma \ref{03}, we obtain that $L(\theta)$ is a Borel radius of
the function $w(z)$.

\section{Open question}
In some literatures, we have known that a radius $L(\theta)$ is a
Borel radius of a $\rho-$order meromorphic function if and only if
there exists a $\varepsilon>0$ such that
\begin{equation}\label{5.1}\limsup\limits_{r\rightarrow1-}\frac{\log
T(r,\Delta(\theta,\varepsilon),f)}{\log\frac{1}{1-r}}=\rho.
\end{equation} And it is easy to prove that if $L(\theta)$ is a Borel
radius of a $\rho-$order algebroid function $w(z)$, then \eqref{5.1}
holds. Here we ask if the converse proposition holds.

\end{document}